\documentclass[9pt,twoside,english,reqno,a4paper]{amsart}
\usepackage{listings,graphicx,amsmath,varioref,amscd,amssymb,color,bm,stmaryrd,amsthm,amsfonts,graphics,geometry,latexsym,pgf,pst-all} 
\theoremstyle{plain}
\usepackage{esint}
\usepackage{amsthm}

\theoremstyle{plain}
\newtheorem{theorem}{Theorem}[section]

\newtheorem{lemma}[theorem]{Lemma}

\newtheorem{corollary}[theorem]{Corollary}

\usepackage{geometry}
\usepackage[colorlinks=false]{hyperref}

\theoremstyle{definition}
\newtheorem{defin}[theorem]{Definition}

\newtheorem{remark}[theorem]{Remark}

\theoremstyle{remark}

\numberwithin{equation}{section}

\DeclareMathOperator{\sgn}{sgn}

\newcommand{\car}[1]{\raise1pt\hbox{$\chi$}_{#1}}

\newcommand{\DM }{\mathcal{DM}^\infty }

\begin{document}
	\title[$1-$laplacian problems with singular first order terms]{Existence of solutions for $1-$laplacian problems with singular first order terms}

	\author[F. Balducci]{Francesco Balducci}
	
	\address{Francesco Balducci
		\hfill \break\indent
		Dipartimento di Scienze di Base e Applicate per l' Ingegneria, Sapienza Universit\`a di Roma
		\hfill \break\indent
		Via Scarpa 16, 00161 Roma, Italy}
	\email{\tt francesco.balducci@uniroma1.it}


	\begin{abstract}
		We prove existence of solutions to the following problem
		\begin{equation*}
			\begin{cases}
				-\Delta_1 u +g(u)|Du|=h(u)f & \text{in $\Omega$,}\\
				u=0 & \text{on $\partial\Omega$,}
			\end{cases}
		\end{equation*}
		where $\Omega \subset \mathbb{R}^N$, with $N\ge2$, is an open and bounded set with Lipschitz boundary, $g$ is a continuous and positive function which possibly blows up at the origin and bounded at infinity and $h$ is a continuous and nonnegative function bounded at infinity (possibly blowing up at the origin) and finally $0 \le f \in L^N(\Omega)$.\\As a by-product, this paper extends the results found where $g$ is a continuous and bounded function. \\We investigate the interplay between $g$ and $h$ in order to have existence of solutions.
	\end{abstract}
	
	\maketitle
	\tableofcontents
	
	\section{Introduction}
	In this paper we study the following nonlinear elliptic  Dirichlet \black problem which is
	\begin{equation}\label{the prob}
		\begin{cases}
			-\Delta_{1}u + g(u)|Du|=h(u)f & \text{in $\Omega$,}\\
			
			u=0 & \text{on $\partial \Omega$,}
		\end{cases}
	\end{equation}
	where $\Omega \subset \mathbb{R}^N$ with $N \ge 2$ is an open bounded set with Lipschitz boundary, $\Delta_1 u:=\operatorname{div}\left(\frac{Du}{|Du|}\right)$ is the $1-$laplacian operator and $g$ is a positive and continuous function on $[0, \infty)$ which possibly blows up at the origin and is bounded at infinity. Finally, $0 \le f \in L^N(\Omega)$ and $h$ is a continuous and nonnegative function on $[0, \infty)$ that is allowed to blow up at the origin and bounded at infinity. We highlight that the case of continuous, bounded and nonmonotone functions $g,h$ is covered by the above assumptions and by \cite{BOP}. 
	\medskip
	\\This type of problems has also been studied as a model for the level set formulation  as explain \black in \cite{HI} for the inverse mean curvature flow, as done initially in \cite{MS}, in order to prove the well-known Penrose inequality for a single black hole.
	\medskip
	\\Our aim is proving the existence of finite energy solutions of \eqref{the prob}; i.e. a function $u \in BV(\Omega)$, which is the space where such problems are naturally built-in in case of smooth nonlinear terms and data.
	\medskip
	\\  Problems involving natural growth gradient terms are widely studied in the literature in the case of $p$-growth, namely \black \begin{equation}\label{prob p}
		\begin{cases}
			-\Delta_p u+g(u)|\nabla u|^p =h(u)f & \text{in $\Omega$,}\\
			u=0 & \text{on $\partial \Omega$.}
		\end{cases}
	\end{equation} In this case, despite the irregularity of the datum \( f \in L^1(\Omega) \), the existence of a finite-energy solution in \( W^{1,p}_0(\Omega) \) is ensured by the presence of the regularizing gradient term, as shown in \cite{BG, BMP}. \black
	\medskip
	\\We want to explore the interplay between the first order and possibly singular absorption term involving $g$ and the zero order and possibly singular nonlinearity $h$ in presence of a datum $f \in L^N(\Omega)$ and the game between the explosion range of $g$ and $h$ at the origin. The problem \eqref{prob p}  was already \black  studied in the case of the laplacian operator (i.e. $p=2$) in \cite{ACLMOP} with $g(s) \sim s^{-\theta}$ with $\theta >0$ near the origin and $h\equiv 1$ where the condition for the existence result of finite energy solutions is $\theta <2$; later the  result \black was extended to $p-$laplacian  case provided $\theta<p$ in \cite{WW}. \black  
	\medskip
	\\Let us briefly discuss the literature concerning the existence of solution for the $1-$laplacian operator,  one usually solves the corresponding problem with the operator $p-$laplacian, finds uniform estimates in $p$, and then lets $p$ tend to $1$. 
	\\In order to give sense to the $1-$laplacian operator in \cite{ABCM} the authors used the Anzellotti theory of pairings $(z, Du)$ of $L^\infty-$divergence$-$measure vector fields $z$ and the gradient of a $BV$ function $u$; the vector field $z \in \DM(\Omega)$ (see Section 2 for more details) is such that $\|z\|_{L^\infty(\Omega)^N}\le1$ and $(z,Du)=|Du|$, in this way $z$ plays the role of the singular quotient $\frac{Du}{|Du|}$. Furthermore, the datum on the boundary is not attained in the classical sense, as it involves the trace of the vector field
	$z$ on the boundary $\partial \Omega$ (for more details, see Remark \ref{rem post def}). \black
	\medskip \\If $g\equiv 0$ and $h \equiv 1$, it is well known that $u \equiv 0$ if $\|f\|_{L^N(\Omega)} < \mathcal{S}^{-1}_1$,  where  $\mathcal{S}_1$ is the constant defined in Theorem  \ref{embedding} (for more details see \cite{CT}); instead in \cite{MST2} the authors proved that $u \equiv \infty$ in a subset of the domain with positive Lebesgue measure if $\|f\|_{L^N(\Omega)}>\mathcal{S}^{-1}_1$. 
	\\The case with $g \equiv 0$ and $h(s)\sim s^{-\gamma}$ with $\gamma>0$ near the origin is discussed in \cite{DCGOP}. The solution of this problem exists when $\|f\|_{L^N(\Omega)} <\left( \mathcal{S}_1h(\infty) \right)^{-1}$ \black  and in particular it is shown that $u>0$ a.e. if $h(0)=\infty$.
	\\The case with $h \equiv 1$ and $g$ is a nonnegative and continuous function is analyzed in \cite{LTS}. The presence of a gradient term introduced some regularizing effects such as there is no jump part in $Du$ and no smallness condition for the norm of $\|f\|_{L^N(\Omega)}$ is needed in order to obtain existence of solutions. Finally, the data is pointwise  assumed \black $\mathcal{H}^{N-1}$ on the boundary.
	\medskip
	\\In conclusion, in \cite{BOP} we can see all the regularizing effects described previously due to the presence of the functions $g$ and $h$, we highlight that there exists a solution in $BV(\Omega)$ because $\gamma \le 1$ and the datum is pointwise assumed on the boundary because $g$ is a continuous, positive and bounded function.
	\medskip
	\\Our aim is generalizing the condition on $\theta$ when $g(s) \sim s^{-\theta}$ near the origin with the presence of a singular zero$-$order term on the right-side in the context of $1-$laplacian operator; the conditions found on $\theta$ and $\gamma$ generalize the ones found in \cite{ACLMOP, WW}, in fact we assume that $0\le \theta < 1$ to gain the integrability of the function $g$ near the origin and $\theta + \gamma \le 1$, moreover there is no smallness assumption on $\|f\|_{L^N(\Omega)}$ for the existence of a finite energy solution $u \in BV(\Omega)$   (for further details on the significance of these hypotheses, refer Remark \ref{oss assumptions} ). \black
	\medskip
	\\Let us summarize the contents present of this paper. In Section $2$ we introduce the notation and we give preparatory tools. In Section $3$ we   present \black the main assumptions and results. Finally, in Section $4$, we discuss the case of a nonnegative $f \in L^N(\Omega)$.

	\section{Notation and preparatory tools}
	In this paper, we denote by $\Omega \subset \mathbb{R}^{N}$, with $N \ge 2$, an open bounded set with Lipschitz boundary. We indicate with $|E|$ the $N-$dimensional Lebesgue measure of a set $E$,  whereas \black $\mathcal{H}^{N-1}(\partial E)$ means the $(N-1)-$dimensional Hausdorff measure.
	\\$\mathcal{M}(\Omega)$ is the space of Radon measures with finite total variation over $\Omega$. $\mathcal{M}_{\rm{loc}}(\Omega)$ is its local counterpart, i.e. the space of Radon measures with locally finite total variation on $\Omega$.
	\medskip
	\\Let us introduce the truncation functions, fixed $k>0$: 
	\begin{equation}\label{trc fun}
		T_k(s):=\begin{cases}
			-k & \text{if $s < -k$,} \\
			s & \text{if $|s|\le k$,} \\
			k & \text{if $s>k$, }
		\end{cases}
	\end{equation}
	and $G_k:\mathbb{R} \to \mathbb{R}$ is \begin{equation}\label{def G k}
		G_k(s):=s-T_k(s).
	\end{equation}
	Furthermore, for a fixed $\delta>0$, we define another type of truncation function $V_\delta : [0, \infty) \to [0,1]$
	\begin{equation}\label{def v d}
		V_\delta(s):=\begin{cases}
			1 & 0\le s \le \delta, \\
			\frac{2\delta -s}{\delta} & \delta<s < 2 \delta, \\
			0 & s \ge 2\delta.
		\end{cases}
	\end{equation}
	\\In the entire paper,  we use the following notation
	$$\int_{\Omega}f := \int_{\Omega} f(x) \, \ensuremath{\mathrm d}x.$$
	\\Moreover, we denote by $C^0_b(\Omega)$ the space of continuous and bounded real functions.
	\\Finally, by $C$ we mean several positive constants whose value change from line to line or on the same line whose value does not depend on the indexes of introduced sequence, but it only depends on the data. In particular, we will not relabel an extracted subsequence. 
	\subsection{Essential properties of $BV$ functions} 
	The set of bounded variation functions is 
	$$BV(\Omega):= \{u \in L^1(\Omega) : Du \in \mathcal{M}(\Omega)^N\},$$   where $Du$ is the distributional gradient. \black
	$BV(\Omega)$ is a Banach space endowed with the norm 
	$$\||u|\|_{BV(\Omega)}:=\|u\|_{L^1(\Omega)}+ \int_\Omega|Du|,$$ where $|Du|$ is the total variation of the vector measure $Du$. However in the paper we use the following equivalent norm \black
	$$\|u\|_{BV(\Omega)}:=\int_{\partial \Omega} |u| \, \ensuremath{\mathrm d}\mathcal{H}^{N-1} + \int_{\Omega} |Du|\, .$$ 
	\medskip
	\\Furthermore we recall the following embedding Theorem.
	\begin{theorem}\label{embedding}
		The embeddings $BV(\Omega) \hookrightarrow L^p(\Omega)$ are compact for every $1 \le p < \frac{N}{N-1}$. The embedding $BV(\Omega) \hookrightarrow L^{\frac{N}{N-1}}(\Omega)$ is continuous and $\mathcal{S}_1$ is the best constant of this embedding,  i.e. \black for every $u \in BV(\Omega)$, it holds
		\begin{equation}\label{f embedding}
			\|u\|_{L^{\frac{N}{N-1}}(\Omega)} \le \mathcal{S}_1 \|u\|_{BV(\Omega)}.
		\end{equation}
	\end{theorem}
	Let us state the compactness result for $BV-$functions (\cite[Theorem 3.23]{AFP}).
	\begin{theorem}
		Consider a sequence of functions $u_n$ \black  uniformly bounded with respect to $n\in \mathbb{N}$ in $\|\cdot \|_{BV(\Omega)}$. Then there exists $u \in BV(\Omega)$ such that $u_n \to u$ strongly in $L^1(\Omega)$ and $Du_n \rightharpoonup Du$ *$-$weakly in $\mathcal{M}(\Omega)$.
	\end{theorem}
	Now we state \black another important result for $BV-$functions (see for instance \cite[Proposition 3.6]{AFP}).
	\begin{lemma}\label{sci BV}
		Let consider a sequence $u_n \in BV(\Omega)$ such that $u_n \to u$ strongly in $L^1(\Omega)$ with $u \in BV(\Omega)$. Then \begin{equation*}
			\int_\Omega |Du| \varphi \le \liminf_{n \to \infty} \int_\Omega |Du_n|\varphi \quad \text{for all $0\le \varphi \in C^1_c(\Omega)$. \black}
		\end{equation*}  
	\end{lemma}
	\medskip
	For a function $u \in L^1_{\rm{loc}}(\Omega)$, we denote with $L_u$ the set of its Lebesgue points, with $S_u=\Omega \setminus L_u$ and with $J_u$ the jump set. In particular if $u \in BV(\Omega)$ the set $S_u \setminus J_u$ is $\mathcal{H}^{N-1}-$negligible, hence $u$ is well defined $\mathcal{H}^{N-1}-$a.e. In this case $u$ can be identified with the precise representative $u^*$ which is $$u^*(x):=\begin{cases}
		\tilde{u}(x) & \text{if $x \in L_u$,}\\
		\frac{u^+(x) + u^-(x)}{2} & \text{if $x \in J_u$,}
	\end{cases}$$ where $\tilde{u}$ is the Lebesgue's representative of $u$, $u^+$ and $u^-$  are the approximate limits of $u$.
	\\When $D^j u =0$, it means that $\mathcal{H}^{N-1}(J_u)=0$ or, equivalently, that $Du= \tilde{D}u$ where $\tilde{D}u$ is the absolutely continuous part of $Du$ with respect to the Lebesgue measure. As a consequence, we will denote the precise representative with $u$ instead of $u^*$, without ambiguity, when we integrate against a measure absolutely continuous with respect to $\mathcal{H}^{N-1}$. 
	\medskip
	\\ Finally, let us remember an important property for $BV$- functions (for more details see \cite[Theorem 5.14.4]{Z}) which ensures that they are finite $\mathcal{H}^{N-1}$- a.e. $x \in \Omega$.
	\begin{lemma}
		Let $u \in BV(\Omega)$, then it holds
		\begin{equation}\label{BV finita}
			\lim_{r \to 0} \frac{1}{|B_r(x_0)|} \int_{B_r(x_0)}u(x)=u(x_0),
		\end{equation}
		for $\mathcal{H}^{N-1}$-a.e. $x_0 \in \Omega$, where $B_r(x_0):=\left\{x \in \Omega \, :\, |x-x_0|<r\right\}$.
	\end{lemma}
	
	\black

	For more details on $BV-$functions see \cite[Chapter 3]{AFP}.
	\subsection{The Anzellotti-Chen-Frid theory}
	Let us present the $L^\infty-$divergence$-$measure vector fields theory discussed, for the first time, in \cite{A} and \cite{CF}. We introduce the space
	$$\DM(\Omega):=\{z \in L^\infty(\Omega)^N : \operatorname{div}z \in \mathcal{M}(\Omega)\},$$ and its local version $\DM_{\rm{loc}}(\Omega)$ which is the set of bounded vector fields with divergence in $\mathcal{M}_{\rm{loc}}(\Omega)$. First we remember that for $z \in \DM(\Omega)$ the distributional divergence $\operatorname{div}z$ is absolutely continuous with respect to $\mathcal{H}^{N-1}$. \medskip
	\\	In \cite{A} Anzellotti introduced the distribution $(z,Dv): C^1_c(\Omega) \to \mathbb{R}$ such that
	\begin{equation}\label{dist1}
		\langle(z,Dv),\varphi\rangle:=-\int_\Omega v^*\varphi\operatorname{div}z-\int_\Omega
		vz\cdot\nabla\varphi,
	\end{equation}
	where $z \in \DM(\Omega)$ and $v \in BV(\Omega)\cap C^0_b(\Omega)$. Subsequently, different authors extended the previous definition of pairing to vector fields in $z \in \DM(\Omega)$ and functions in $v \in BV(\Omega) \cap L^\infty(\Omega)$, since $v^* \in L^\infty(\Omega, \operatorname{div}z)$ (see for instance \cite{C, MST}). Furthermore formula \eqref{dist1} is well posed if $z \in \DM_{\rm{loc}}(\Omega)$ and $v \in BV_{\rm{loc}}(\Omega) \cap L^1_{\rm{loc}}(\Omega, \operatorname{div}z)$, as proven in \cite{DCGS}; the authors showed that
	$$|\langle (z, Dv), \varphi \rangle| \le \|\varphi\|_{L^\infty(A)} \|z\|_{L^\infty(A)^N} \int_A |Dv|\,,$$
	for all open sets $A \subset \subset \Omega$ and for all $\varphi \in C^1_c(A)$. Moreover, it holds  
	\begin{equation}\label{finitetotal1}
		\left| \int_B (z, Dv) \right|  \le  \int_B \left|(z, Dv)\right| \le  ||z||_{L^\infty(A)^N} \int_{B} |Dv|\,,
	\end{equation}
	for all Borel sets $B$ and for all open sets $A$ such that $B \subset A \subset \Omega$, which means that the measure $(z,Dv)$ is absolutely continuous with respect to $|Dv|$. 
	\medskip
	\\Further, under the same assumptions for $z$ and $v$ we indicate by $\lambda(z, Dv, x)$ the Radon-Nikod\'ym derivative of $(z,Dv)$ with respect to $|Dv|$, hence we can affirm that $$(z,Dv)=\lambda(z,Dv,x)|Dv| \quad \text{as measures in $\Omega$.}$$
	Let us highlight that, if $z \in \DM_{\rm{loc}}(\Omega)$ and $v \in BV_{\rm{loc}}(\Omega) \cap L^\infty(\Omega)$
	\begin{equation}\label{Leibniz}
		\operatorname{div}(vz)=(z,Dv)+v^* \operatorname{div}z \quad \text{as measures in $\Omega$,}
	\end{equation}
	where we underline that, as a consequence, $vz \in \DM_{\rm{loc}}(\Omega)$.
	\medskip
	\\In \cite{A} it is shown that every $z \in \DM(\Omega)$ has a weak \black trace on $\partial \Omega$ of its normal component denoted as $[z,\nu]$, where $\nu(x)$ is the outward normal unit vector defined for $\mathcal{H}^{N-1}-$a.e. $x \in \partial \Omega$. Also we recall that
	\begin{equation*}
		\|[z, \nu]\|_{L^\infty(\partial \Omega)} \le \|z\|_{L^\infty(\Omega)^N},
	\end{equation*}
	and if we take $v \in BV(\Omega)\cap L^\infty(\Omega)$, it holds  \begin{equation}\label{des3}
		v[z, \nu]=[vz, \nu] \quad \text{$\mathcal{H}^{N-1}-$a.e. on $\partial \Omega$,}
	\end{equation} as proved in \cite{C}.
	\medskip
	\\Now we are able to state the generalized Gauss-Green formula for $L^\infty-$divergence$-$measure vector fields which in the form we present is proved in \cite{DCGS}. \black
	\begin{lemma}
		Let $z \in \DM_{\rm{loc}}(\Omega)$ and $v \in BV(\Omega) \cap L^\infty(\Omega)$ such that $v^* \in L^1(\Omega, \operatorname{div} z)$. Then $vz \in \DM(\Omega)$ and the following formula holds:
		\begin{equation}\label{GG form}
			\int_{\Omega} v^* \, \ensuremath{\mathrm d} \operatorname{div} z + \int_{\Omega}(z,Dv) = \int_{\partial \Omega} [vz,\nu] \, \ensuremath{\mathrm d}\mathcal{H}^{N-1}.
		\end{equation}
	\end{lemma}
	\medskip
	For our scope, we also need the following definition of pairing measures, as defined in \cite{MS}. Let us consider $\beta : \mathbb{R} \to \mathbb{R}$ a locally Lipschitz function and let $v \in BV_{\text{loc}}(\Omega)$. Let us define
	\begin{equation*}\label{der grat}
		\displaystyle \beta(v)^{\#} :=\begin{cases} \displaystyle \frac{1}{v^{+} -v^-}\int_{v^-}^{v^+} \beta(s) \ ds  & \text{if  $x\in J_v$,}\\
			\beta(v) & \text{otherwise.}
			\ \end{cases}
	\end{equation*}
	We emphasize that $\beta(v)^\#$ coincides with $\beta(v)^*$ on the jump set of the function $v$ if and only if $\beta(s)=s$. Let $z \in \DM_{\rm{loc}}(\Omega)$ and $v \in BV_{\rm{loc}}(\Omega)$ satisfy $\beta(v) \in BV_{\rm{loc}}(\Omega) \cap L^\infty_{\rm{loc}}(\Omega)$. We introduce the distribution $\left(z, D\beta(v)^\#\right): C^1_c(\Omega) \to \mathbb{R}$ such that
	\begin{equation}\label{mis con grat}
		\langle(z,D\beta(v)^{\#}),\varphi\rangle:=-\int_\Omega \beta(v)^{\#}\varphi\operatorname{div}z-\int_\Omega
		\beta(v)z\cdot\nabla\varphi.
	\end{equation}
	Such pairing constitutes a well-defined measure (refer to, for instance, \cite[Lemma 2.5]{MS}), which is absolutely continuous with respect to $\mathcal{H}^{N-1}$ and 
	\begin{equation}\label{dis mis con grat}
		\left|\int_B (z,D\beta(v)^{\#}) \right|\le \int_B |(z,D\beta(v)^{\#})| \le \|z\|_{L^\infty(A)^N} \int_B |D \beta(v)|,
	\end{equation}
	for all Borel sets $B$ and for all open sets $A$ such that $B \subset A \subset \Omega$.
	\medskip
	\\Let us finally state the chain rule formula as given in \cite[Theorem 3.99]{AFP}.
	\begin{lemma}\label{chain rule}
		Let $v \in BV(\Omega)$ and let $\Phi: \mathbb{R} \to \mathbb{R}$ be a Lipschitz function. Then $w=\Phi(v) \in BV(\Omega)$ and \begin{equation}\label{eq chain}
			Dw=\Phi'(v)^\# Dv.
		\end{equation}
		In particular, if $D^jv=0$, then \begin{equation}\label{eq chain w j}
			\tilde{D}w =\Phi'(v) \tilde{D}v.
		\end{equation}
	\end{lemma}
	\medskip
	Now let us assert two properties of the pairing defined in \eqref{dist1} for bounded variation functions without jump part (see for instance \cite[Lemma 2.6]{GMP}).
	\begin{lemma}
		Let $z \in \DM(\Omega)$ and $u,v \in BV(\Omega) \cap L^\infty(\Omega)$ such that $D^j u=D^j v=0$. Then
		\begin{equation}\label{der prod}
			(z,D(uv))=u(z,Dv)+v(z,Du)=(uz,Dv)+(vz,Du).
		\end{equation}
	\end{lemma}
	\medskip
	We conclude this section with a lemma which is an improvement of the one in \cite{O}; it is a regularity result for a vector field in $z \in \DM_{\rm{loc}}(\Omega)$.
	\begin{lemma}\label{lemma estensione}
		Let $0 \le \tilde{f} \in L^1_{\rm{loc}}(\Omega)$, let $\sigma \in \mathcal{M}(\Omega)$ a measure absolutely continuous with respect to $\mathcal{H}^{N-1}$ and let $z \in \mathcal{DM}^{\infty}_{\rm{loc}}(\Omega)$ such that \begin{equation}\label{= generica}
			-\operatorname{div}z + \sigma = \tilde{f} \, \text{in $\mathcal{D}'(\Omega)$,}
		\end{equation}
		then \begin{equation}\label{divz finita}
			\operatorname{div}z \in \mathcal{M}(\Omega).
		\end{equation}
\end{lemma}
\begin{proof}
We choose $0 \le v \in W^{1,1}_0(\Omega) \cap C^0(\Omega) \cap L^{\infty}(\Omega)$ and we let $\varphi_n \in C^1_c(\Omega)$ be a sequence of nonnegative functions converging to $ v$ strongly in $W^{1,1}_0(\Omega)$. Picking out $\varphi_n$ as a test function \black in \eqref{= generica}, we obtain
\begin{equation*}
	\int_{\Omega} z \cdot \nabla \varphi_n + \int_{\Omega}  \varphi_n \sigma=\int_{\Omega}\tilde{f} \varphi_n.
\end{equation*}
Our aim is taking the limit as $n$ tends to infinity in the previous equality. 
\\For the first term on the left-hand side, it is sufficient observing that $z \in L^{\infty}(\Omega)^N$ and $\varphi_n \to v$ strongly in $W^{1,1}_0(\Omega)$. 
\\In the second one, we pass to the limit through the Lebesgue Theorem because $\varphi_n \to v$ $\sigma-$a. e., since  $\sigma \in \mathcal{M}(\Omega)$ is absolutely continuous with respect to $\mathcal{H}^{N-1}$ and $\varphi_n \to v$ strongly in $W^{1,1}_0(\Omega)$. 
\\For the right-hand side using the Fatou Lemma, we gain
\begin{equation}\label{sol con W^1,1_0}
	\int_{\Omega}\tilde{f} v \le \int_{\Omega}z \cdot \nabla v + \int_{\Omega}v \sigma .
\end{equation}
Now we take $\tilde{v} \in W^{1,1}(\Omega) \cap C^0(\Omega) \cap L^\infty(\Omega)$ and by virtue of a Gagliardo Lemma (see \cite[Lemma 5.5]{A}), there exists $w_n \in W^{1,1}(\Omega) \cap C^0(\Omega) \cap L^\infty(\Omega)$ such that 
\begin{itemize}
	\item $w_n|_{\partial \Omega}=\tilde{v}|_{\partial \Omega}$,
	\item $\|w_n\|_{L^{\infty}(\Omega)} \le \| \tilde{v}\|_{L^{\infty}(\partial\Omega)}$,
	\item $\int_{\Omega} | \nabla w_{n}| \le \int_{\partial \Omega} \tilde{v} \,\ensuremath{\mathrm d}\mathcal{H}^{N-1} + \frac{1}{n}$,
	\item $w_n \to 0$ a.e. in $\Omega$.
\end{itemize}
Taking \black $|\tilde{v} - w_n| \in W^{1,1}_0(\Omega) \cap C^{0} (\Omega) \cap L^\infty(\Omega)$  as a test function \black in \eqref{sol con W^1,1_0}, we get
\begin{eqnarray*}
	\begin{aligned}
		&\int_{\Omega} \tilde{f} | \tilde{v} -w_n|\le \int_{\Omega} z \cdot \nabla|\tilde{v}-w_n| +\int_{\Omega}  |\tilde{v}-w_n| \sigma \\ \le 
		&\|z\|_{L^\infty(\Omega)^N} \left(\int_{\Omega}|\nabla \tilde{v}| + \int_{\Omega}|\nabla w_n|\right) + 2 |\sigma|(\Omega)\|\tilde{v}\|_{L^{\infty}(\Omega)} \\ \le 
		& \|z\|_{L^\infty(\Omega)^N} \left(\int_{\Omega}|\nabla \tilde{v}| + \int_{\partial \Omega} \tilde{v} \, \ensuremath{\mathrm d}\mathcal{H}^{N-1} + \frac{1}{n}\right)+2 |\sigma|(\Omega)\|\tilde{v}\|_{L^{\infty}(\Omega)}.
	\end{aligned}
\end{eqnarray*}
Taking limit as $n$ tends to infinity, from the Fatou Lemma, we have
\begin{equation}\label{dis per l1}
	\int_{\Omega}  \tilde{f} \tilde{v} \le \|z\|_{L^\infty(\Omega)^N} \left(\int_{\Omega}|\nabla \tilde{v}| + \int_{\partial \Omega}\tilde{v} \, \ensuremath{\mathrm d}\mathcal{H}^{N-1}\right) + 2 |\sigma|(\Omega)\|\tilde{v}\|_{L^{\infty}(\Omega)}.
\end{equation}
Then, fixing $\tilde{v} \equiv 1$ in \eqref{dis per l1}, one gets that $\tilde{f} \in L^1(\Omega)$. Therefore this implies \eqref{divz finita}.
\end{proof}
\begin{remark}\label{oss estensione}
As a direct consequence of Lemma \ref{lemma estensione}, we can extend the space of test function to $BV(\Omega)\cap L^\infty(\Omega)$ for the equation \eqref{= generica}, it can show through a density argument, it holds
$$\int_\Omega (z,Dv) -\int_{\partial \Omega} v[z,\nu] \ensuremath{\mathrm d}\mathcal{H}^{N-1} +\int_\Omega v^* \sigma=\int_\Omega \tilde{f} v \quad \text{for all $v \in BV(\Omega) \cap L^\infty(\Omega)$.}$$
\end{remark}
\section{Main assumptions and results for $f>0$}
In this section we deal with existence of solutions to
\begin{equation}\label{problema}
\begin{cases}
	- \Delta_{1}u+g(u)|Du|=h(u)f & \text{in $\Omega$,}\\
	u=0 &\text{on $\partial \Omega$,}
\end{cases}
\end{equation}
where $\Omega \subset \mathbb{R}^{N}$, with $N \ge 2$, is an open and bounded set with Lipschitz boundary, $\Delta_{1}u:=\operatorname{div}\left(\frac{Du}{|Du|}\right)$ is the $1-$laplacian operator. Firstly, we study the case with $0<f \in L^N(\Omega)$. \medskip \\The function $g:[0,\infty) \to (0,\infty]$ \black is continuous and such that \begin{equation}\label{g in zero}
\exists 0 \le \theta < 1, c_{1},s_{1}>0: g(s) \le \frac{c_{1}}{s^{\theta}} \text{ for all $s \le s_{1}$,} 
\end{equation} and
\begin{equation}\label{g non nulla}
\liminf_{s \to \infty} g(s)>0.
\end{equation}
\medskip
Moreover $h:[0,\infty) \to [0,\infty]$ \black is a continuous function such that $h(0)>0$ and
\begin{equation}\label{h in zero}
\exists 0 \le \gamma \le 1, c_{2},s_{2}>0: h(s) \le \frac{c_{2}}{s^{\gamma}} \text{ for all $s \le s_{2}$,} 
\end{equation}
and we require \begin{equation}\label{h g infinito}
h,g \in  C^0_b([\delta, \infty)), \forall \delta>0.
\end{equation}
\medskip
Furthermore we introduce the function $\Gamma : \mathbb{R} \to \mathbb{R}$
\begin{equation}\label{prim di g}
\Gamma(s):=\int_0^s g(\sigma) \, \ensuremath{\mathrm d}\sigma.
\end{equation}
\medskip
\\We set the following quantities which are also useful in the next proofs
$$g_k(\infty):=\sup_{s \in [k, \infty)}g(s) \quad \text{and} \quad g(\infty):=\limsup_{s \to \infty} g(s) < \infty,$$
$$h_k(\infty):=\sup_{s \in [k, \infty)}h(s) \quad \text{and} \quad h(\infty):=\limsup_{s \to \infty} h(s)< \infty.$$
\medskip
\\Let us highlight that the classical case $g,h \equiv 1$ is covered by the above assumptions as $\gamma, \theta$ could be zero.
\medskip
\\Now we precise how the concept of distributional solution for problem \eqref{problema} is intended.
\begin{defin}\label{sol}
Let $0<f\in L^N(\Omega)$. A nonnegative $u\in BV(\Omega)$ is a distributional solution to \eqref{problema} if $D^j u = 0$, $g(u) \in L^1_{\rm{loc}}(\Omega, |Du|)$ and $h(u)f \in L^1_{\rm{loc}}(\Omega)$ and if there exists a vector field $z \in \mathcal{DM}^{\infty}_{\rm{loc}}(\Omega)$, with $\|z\|_{L^\infty(\Omega)^N} \le 1$ such that
\begin{equation}\label{sol 1}
	- \operatorname{div}z + g(u)|Du| = h(u) f \quad \text{as measures in $\Omega$,}
\end{equation}  
\begin{equation}\label{sol 2}
	(z, DT_k(u))=|DT_k(u)| \,\, \text{in $\mathcal{D}'(\Omega)$ for any $k>0$,}
\end{equation}
and \begin{equation}\label{sol 3}
	u(x)=0 \, \, \text{for $\mathcal{H}^{N-1}-$a.e. $x \in \partial \Omega$.}
\end{equation}
\end{defin}
\begin{remark}\label{rem post def}
Let us briefly discuss Definition \ref{sol}. Firstly, formulas \eqref{sol 1} and \eqref{sol 2} represent the weak manner in which \( z \) plays the role of as the singular quotient \( \frac{Du}{|Du|} \).
\medskip
\\Furthermore, \eqref{sol 3} underscores that the boundary condition is assumed to hold pointwise. This is intrinsically linked to the presence of the gradient term, as will be examined later. It is now widely recognized that solutions to \( 1 \)-Laplace Dirichlet problems typically do not satisfy the pointwise enforcement of boundary conditions when \( g \equiv 0 \) (see, for instance, \cite{DCGOP, LTS, O}). In such cases, the weaker condition 
\[
u\left([z,\nu] + \sgn(u)\right) = 0 \quad \mathcal{H}^{N-1}\text{-a.e. on } \partial \Omega
\]
is usually imposed. This condition essentially asserts that either \( u \) has a zero trace, or the weak trace of the normal component of \( z \) attains the minimal possible slope at the boundary.
\medskip
\\Finally, we highlight that if \( h(0) = \infty \), since \( h(u)f \in L^1_{\rm{loc}}(\Omega) \), it follows that \( \{u = 0\} \subseteq \{f = 0\} \). Therefore, given \( f > 0 \), we deduce that \( u > 0 \).

\end{remark}
At this point we state the main result of this section.
\begin{theorem}\label{teo princ}
Let $g$ be positive and satisfy \eqref{g in zero}, \eqref{g non nulla} and \eqref{h g infinito}, let $h$ satisfy \eqref{h in zero} and \eqref{h g infinito} such that  $0\le \theta < 1$, $0 \le \gamma \le 1$, and $\theta+\gamma \le 1$  and let $0<f \in L^N(\Omega)$. Then there exists a solution to problem \eqref{problema} in the sense of Definition \ref{sol}.
\\Moreover if $\|f\|_{L^N(\Omega)}<(\mathcal{S}_1 h(\infty))^{-1}$, then $u \in L^\infty(\Omega) $.
\end{theorem}
\begin{remark}\label{oss assumptions} 
Let us spend few words on the statements of Theorem \ref{teo princ}.
\medskip
\\Initially, the condition \eqref{g non nulla} is necessary for the existence of the function \( u \) (for more details see Lemma \ref{lemma stima}).
\medskip
\\We emphasize that the assumption regarding the exponent $\theta$ is necessary for the integrability of the function $g$. In contrast, the statement concerning $\gamma$ is natural for solution to be globally in $BV(\Omega)$, in fact when $g \equiv 0$ and $\gamma>1$ the solutions are only locally in $BV(\Omega)$ (for more details see \cite{DCGOP}). Additionally, the condition $\theta + \gamma \le 1$ and $f \in L^N(\Omega)$ are crucial  to have $\Gamma(u) \in BV(\Omega)$ (see Lemma \ref{lemma stima}), \black where $\Gamma$ is the function defined in \eqref{prim di g}.
\medskip
\\We underline that the positivity of the function \( g \) is necessary for \( \Gamma \)  to be increasing, which is essential to prove that \( D^j u = 0 \).
\medskip
\\Finally, if $\|f\|_{L^N(\Omega)}< (\mathcal{S}_1 h(\infty))^{-1}$, we can show that $u \in L^\infty(\Omega)$ through a Stampacchia's method, as a consequence fixing $k> \|u\|_{L^\infty(\Omega)}$, \eqref{sol 2} becomes 
$$\left(z, Du\right)=|Du| \quad \text{as measures in $\Omega$.}$$ 

\end{remark}
\subsection{Approximation scheme and existence of a limit function}
We find a solution of the problem \eqref{problema} by approximation, let us consider \begin{equation}\label{prob approx}
\begin{cases}
	- \Delta_{1}u_{n}+g_{n}(u_n)|Du_{n}|=h(u_{n})f & \text{in $\Omega$,}\\
	u_n=0 &\text{on $\partial \Omega$,}
\end{cases}
\end{equation} 
where $g_{n}(s):=T_{n}(g(s))$ for any $s \in [0, \infty)$, $n \in \mathbb{N}$ and $T_n$ is the truncation function defined in \eqref{trc fun}.
\medskip
\\It follows from \cite[Theorem 4.4]{BOP} that there exists a solution to \eqref{prob approx}, i.e. there exists $z_n \in \mathcal{DM}^{\infty}_{\rm{loc}}(\Omega)$ with $\|z_n\|_{L^{\infty}(\Omega)^{N}} \le 1$ and a nonnegative function $u_{n} \in BV(\Omega)$ such that $D^j u_n=0$, $g_n(u_n) \in L^{1}(\Omega, |Du_n|)$ and $h(u_n)f \in L^{1}_{\rm{loc}}(\Omega)$ such that
\begin{equation}\label{sol approx 1}
- \operatorname{div}z_n + g_n(u_n)|Du_n|=h(u_n)f \quad \text{as measures in $\Omega$,}
\end{equation}
\begin{equation}\label{sol approx 2}
(z_n, DT_k (u_n))=|DT_k(u_n)| \text{ as measures in $\Omega$, for any $k>0$,}
\end{equation}
and \begin{equation}\label{sol approx 3}
u_n(x)=0 \text{ for $\mathcal{H}^{N-1}-$a.e. $x\in \partial \Omega$.}
\end{equation}
\medskip
Now we estimate $u_n$ and $\Gamma_n(u_n)$ in $BV(\Omega)$, where, for every $n \in \mathbb{N}$, $\Gamma_n: \mathbb{R} \to \mathbb{R}$ is such that
\begin{equation*}
\Gamma_n(s):=\int_0^s g_n(\sigma) \, \ensuremath{\mathrm d}\sigma.
\end{equation*}
\begin{lemma}\label{lemma stima}
Let $g$ be positive and satisfying \eqref{g in zero}, \eqref{g non nulla} and \eqref{h g infinito}. Let $h$ satisfy \eqref{h in zero} and \eqref{h g infinito} such that $0\le \theta < 1$, $0 \le \gamma \le 1$ and $\theta+\gamma \le 1$,  and let $0<f \in L^N(\Omega)$. Finally let $u_n$ be a solution to \eqref{prob approx}. \\Then $u_n$, $\Gamma_n(u_n)$ are uniformly bounded with respect to $n$ in $BV(\Omega)$, there exists $C_1,C_2>0$ independent of $n \in \mathbb{N}$ such that
\begin{equation}\label{u_n in BV}
	\| u_n\|_{BV(\Omega)} \le C_1,
\end{equation}  and
\begin{equation}\label{Gamma_n BV}
	\|\Gamma_n(u_n)\|_{BV(\Omega)} \le C_2,
\end{equation}

and $h(u_n)f$ is uniformly bounded with respect to $n \in \mathbb{N}$ in $L^1(\omega)$, for every $\omega \subset \subset \Omega$. \black \\Moreover if $\|f\|_{L^N(\Omega)}<(\mathcal{S}_1h(\infty))^{-1}$, then $u_n$ is uniformly bounded in $L^\infty(\Omega)$. 
\end{lemma}

\begin{proof}
We first show that $u_n$ is bounded in $BV(\Omega)$ with respect to $n \in \mathbb{N}$. Without losing \black generality, we can assume that $s_1=s_2=\overline{s}$ and $c_1=c_2=\overline{c}$, where  $s_1,s_2,c_1,c_2$ are the constants introduced in \eqref{g in zero} and \eqref{h in zero}. \black
\medskip
\\We choose $T_1(u_n)\in BV(\Omega)\cap L^\infty(\Omega)$ as a test function \black in \eqref{sol approx 1} which is an admissible choice thanks to Remark \ref{oss estensione}.
\\Initially we note that
\begin{equation}\label{I int}
	-\int_{\Omega} \operatorname{div} z_n \, T_1(u_n)\stackrel{\eqref{GG form}, \eqref{sol approx 3}}{=}\int_\Omega \left(z_n, DT_1(u_n)\right) \stackrel{\eqref{sol approx 2}}{=}\int_\Omega |DT_1(u_n)| \stackrel{\eqref{eq chain w j}}{=} \int_{\{u_n \le 1\}} |Du_n|,
\end{equation}
which is possible because $u_n$ is nonnegative. For the integral on the right-hand side, we deduce 
\begin{equation}\label{II int}
	\int_\Omega h(u_n)f T_1(u_n) \stackrel{\eqref{h in zero}}{\le} \int_{\{u_n \le \overline{s}\}}\frac{\overline{c}}{u_n^\gamma}fu_n + \int_{\{u_n > \overline{s}\}}h(u_n)f \stackrel{\eqref{h g infinito}}{\le} \left(\overline{c} \overline{s}^{1-\gamma} +h_{\overline{s}}(\infty)\right) \| f \|_{L^1(\Omega)},
\end{equation} 
where we highlight that the right-hand side is a constant independent of $n \in \mathbb{N}$. 
Putting together \eqref{I int} and \eqref{II int} in \eqref{sol approx 1} with $T_1(u_n)$ as a test function \black, we gain \begin{equation}\label{test T1}
	\int_{\{u_n \le 1\}} |Du_n| + \int_{\{u_n \le 1\}} g_n(u_n)|Du_n|u_n + \int_{\{u_n > 1\}} g_n(u_n)|Du_n| \le C.
\end{equation}
At this point we stress that as a consequence of \eqref{g non nulla}, we can assure that there exists $\eta>0$ such that $$\inf_{s \in (1,\infty)} g(s)>\eta>0,$$
so applying \eqref{h g infinito} in \eqref{test T1}, \black we can affirm that
\begin{equation*}
	\min\left\{1, \inf_{s \in (1, \infty)}g(s)\right\} \left(\int_{\{u_n \le 1\}}|Du_n|+\int_{\{u_n > 1\}}|Du_n|\right) \le C,
\end{equation*} 
hence, we have proven \eqref{u_n in BV}.
	\medskip
	\\Now we focus on proving that $\Gamma_n(u_n)$ is uniformly bounded with respect to $n \in \mathbb{N}$ in $BV(\Omega)$. First of all, we want to show \begin{equation}\label{parring gamma}
		(z_n, D\Gamma_n(T_k(u_n)))=|D\Gamma_n(T_k(u_n))| \quad \text{as measures in $\Omega$.}
	\end{equation}
	From the definition of $\lambda(z_n, \Gamma_n(T_k(u_n)),x)$ we know that
	$$(z_n, D\Gamma_n(T_k(u_n)))= \lambda(z_n, \Gamma_n(T_k(u_n)), x)|D\Gamma_n(T_k(u_n))| \quad \text{as measures in $\Omega$,}$$ by \cite[Proposition 4.5 (iii)]{CDC} for every $k>0$, we gain
	\begin{equation*}
		\lambda(z_n, D\Gamma_n(T_k(u_n)), x)=\lambda(z_n,DT_k(u_n),x)=1 \quad \text{$|D\Gamma_n(T_k(u_n)))|-$a.e.,}
	\end{equation*}
	where we recall that $\lambda(z_n,DT_k(u_n),x)=1$ $|DT_k(u_n)|-$a.e. from \eqref{sol approx 2}, hence we get \eqref{parring gamma}.
	\\Subsequently, we take $\Gamma_n(T_k(u_n))\in BV(\Omega) \cap L^\infty(\Omega)$ with $k> \overline{s}$ as a test function \black in the equation \eqref{sol approx 1}, which is an admissible choice by Remark \ref{oss estensione}, \black first we show that
	\begin{equation}\label{int III}
		-\int_{\Omega} \operatorname{div} z_n \, \Gamma_n(T_k(u_n))\stackrel{\eqref{GG form}, \eqref{sol approx 3}}{=}\int_\Omega \left(z_n, D\Gamma_n(T_k(u_n))\right) \stackrel{\eqref{parring gamma}}{=}\int_\Omega |D\Gamma_n(T_k(u_n))|.
	\end{equation}
	Hence putting \eqref{int III} in \eqref{sol approx 1} with $\Gamma_n(T_k(u_n))$ as a test function \black, we have
	\begin{eqnarray*}
		\int_{\Omega}|D\Gamma_n(T_k(u_n))| &\le& \int_{\Omega}h(u_n)f \Gamma_n(T_k(u_n)) \le \frac{\overline{c}^2}{1-\theta}\int_{\{u_n \le \overline{s}\}} u_n^{1-\theta-\gamma} f \\ &+& h_{\overline{s}}(\infty) g_{\overline{s}}(\infty)\|f\|_{L^N(\Omega)}\|u_n\|_{L^{\frac{N}{N-1}}(\Omega)} \stackrel{\eqref{f embedding}}{\le}   \frac{\overline{c}^2 \overline{s}^{1-\theta-\gamma}}{1-\theta}\|f\|_{L^1(\Omega)} \\&+&  \mathcal{S}_1 h_{\overline{s}}(\infty) g_{\overline{s}}(\infty)  \|u_n\|_{BV(\Omega)} \|f\|_{L^N(\Omega)}\stackrel{\eqref{u_n in BV}}{\le} C\|f\|_{L^N(\Omega)},
	\end{eqnarray*}
	where we have used \eqref{g in zero}, \eqref{h in zero}, \eqref{h g infinito} and Hölder's inequality.
	\\We point out that the previous integral is uniformly bounded with respect to $n \in \mathbb{N}$ and $k > \overline{s}$, since $0\le \gamma \le 1$, $0\le \theta <1$ and $\theta+\gamma \le 1$, hence through Lemma \ref{sci BV}, one can take $k \to \infty$ yielding to 			\begin{equation*}
		\int_{\Omega} |D \Gamma_n(u_n)| \le C\|f\|_{L^N(\Omega)},
	\end{equation*}
	therefore we get \eqref{Gamma_n BV} (recall that $u_n=0 \, \mathcal{H}^{N-1}-$a.e. on $\partial \Omega$ which is guaranteed in \eqref{sol approx 3}).
	\medskip
	\\Now we show that $h(u_n)f$ is uniformly bounded with respect to $n \in \mathbb{N}$ in $L^1(\omega)$ for every $\omega \subset \subset \Omega$. We choose $0 \le \varphi \in C^1_c(\Omega)$ such that $\varphi \equiv 1$ in $\omega$ as a test function in \eqref{sol approx 1}; then it follows from \eqref{Gamma_n BV} that
	\begin{equation}\label{h(u_n)f uniforme}
		\int_\omega h(u_n) f \varphi \le \int_\Omega z_n \cdot \nabla \varphi + \int_\Omega |D\Gamma_n(u_n)| \varphi \le \| \nabla \varphi\|_{L^1(\Omega)^N} + C\|f\|_{L^N(\Omega)} \| \varphi\|_{L^\infty(\Omega)},
	\end{equation} \black
	where we also used that $\|z_n\|_{L^\infty(\Omega)^N}\le 1$.
	\medskip
	\\Finally let us demonstrate that $u_n$ is uniformly bounded with respect to $n$ in $L^\infty(\Omega)$; we stress that the condition on the smallness of norm $\|f\|_{L^N(\Omega)}$ is crucial for this estimate.
	\\Let us take $G_k(T_r(u_n)) \in BV(\Omega) \cap L^\infty(\Omega)$ where $G_k$ is the function defined \eqref{def G k} with $r>k>\overline{s}$ \black as a test function \black in \eqref{sol approx 1}; it is an admissible choice due to Remark \ref{oss estensione}, we obtain
	\begin{equation}\label{mezza stima G_k}
		\int_\Omega (z_n, DG_k(T_r(u_n))) \le \int_{\Omega} h(u_n)f G_k(u_n) \le h_{k}(\infty) \black \, \|f\|_{L^N(\Omega)} \|G_k(u_n)\|_{L^{\frac{N}{N-1}}(\Omega)},
	\end{equation}
	where we also used \eqref{GG form} and \eqref{sol approx 3}.
	Through \cite[Proposition 4.5 (iii)]{CDC}, we deduce that
	\begin{equation*}
		\lambda(z_n, DG_k(T_r(u_n)), x)=\lambda(z_n, DT_r(u_n),x)=1 \quad \text{$|DG_k(T_r(u_n))|-$a.e.}
	\end{equation*}
	since $G_k$ is a non-decreasing function, hence we can affirm that
	\begin{equation}\label{Parring G_k}
		(z_n, DG_k(T_r(u_n))) = |DG_k(T_r(u_n))| \quad \text{as measures in $\Omega$,}
	\end{equation}
	because $|DG_k(T_r(u_n))|$ is an absolutely continuous measure \black with respect to $|DT_r(u_n)|$.
	\\Gathering \eqref{Parring G_k} into \eqref{mezza stima G_k}, it yields \black 
	\begin{equation*}
		\int_{\{k\le u_n \le r\}} |Du_n|=	\int_\Omega |DG_k(T_r(u_n))| \le  h_k(\infty) \black \|f\|_{L^N(\Omega)} \|G_k(u_n)\|_{L^{\frac{N}{N-1}}(\Omega)},
	\end{equation*}
	where we have applied Lemma \ref{chain rule}, since $G_k$, $T_r$ are Lipschitz functions and $D^ju_n = 0$. Due to the Lebesgue Theorem and Theorem \ref{embedding}, taking $r$ at infinity, we have 
	\begin{equation*}
		\mathcal{S}_1^{-1} \|G_k(u_n)\|_{L^\frac{N}{N-1}(\Omega)} \le h_k(\infty) \|f\|_{L^N(\Omega)} \|G_k(u_n)\|_{L^{\frac{N}{N-1}}(\Omega)}.
	\end{equation*} 
	Under the assumption that $\|f\|_{L^N(\Omega)} \mathcal{S}_1 h(\infty) <1$, we can pick $\overline{k}\in \mathbb{N}$ sufficiently large such that $\mathcal{S}_1^{-1} - \|f\|_{L^N(\Omega)} h_{\overline{k}}(\infty)>C>0$ for a constant $C$ which is independent of $n \in \mathbb{N}$. Therefore, we get 
	$$\|G_{\overline{k}}(u_n)\|_{L^{\frac{N}{N-1}}(\Omega)} \le 0,$$
	which means that 
	\begin{equation*}\label{u_n uniforme}
		\|u_n\|_{L^\infty(\Omega)} \le \overline{k},
	\end{equation*}
	where $\overline{k}>0$ is a constant independent of $n \in \mathbb{N}$. 
\end{proof}
\begin{corollary}\label{cor es}
	Under the assumptions of Lemma \ref{lemma stima}, there exists a nonnegative $u \in BV(\Omega)$ such that $u_n$ converges to $u$ (up to subsequence) in $L^q(\Omega)$ for every $q<\frac{N}{N-1}$ and $Du_n$ converges to $Du$ *$-$weakly as measures.
	\\Moreover $\Gamma(u) \in BV(\Omega)$ and $\Gamma_n(u_n)$ converges to $\Gamma(u)$ (up to subsequence) in $L^q(\Omega)$ for every $q < \frac{N}{N-1}$ and $D\Gamma_n(u_n)$ converges to \black $D\Gamma(u)$ *$-$weakly as measures.
	\\Finally if $\|f\|_{L^N(\Omega)}<(\mathcal{S}_1h(\infty))^{-1}$, then $u$ is also bounded.
\end{corollary}
\subsection{Identification of the vector field $z$}
Now we show the existence of the vector field $z$.
\begin{lemma}\label{lemma dis}
Under the assumptions of Lemma \ref{lemma stima}, it holds $h(u)f \in L^1_{\rm{loc}}(\Omega)$ and there exists $z \in \mathcal{DM}^{\infty}_{\rm{loc}}(\Omega)$ such that
\begin{equation}\label{= per Gamma e z}
	-\operatorname{div}\left(z \ensuremath{\mathrm e}^{-\Gamma(u)}\right)=h(u)f \ensuremath{\mathrm e}^{-\Gamma(u)} \, \text{as measures in $\Omega$,}
\end{equation}
and
\begin{equation}\label{dis per z}
	-\operatorname{div}z+|D\Gamma(u)| \le h(u)f \quad \text{as measures in $\Omega$,}
\end{equation}
Moreover $u>0$ a.e. in $\Omega$ and $D^j u = 0$.
\end{lemma}
\begin{proof}
Let us observe that $h(u)f \in L^1_{\rm{loc}}(\Omega)$ simply follows from \eqref{h(u_n)f uniforme} after an application of the Fatou Lemma.
\\Firstly we consider $z_n \in \mathcal{DM}^\infty(\Omega)$, the vector field which plays the role of $\frac{Du_n}{|Du_n|}$. We remember that $z_n$ is uniformly bounded with respect to $n$ in $L^\infty(\Omega)^N$, since $\|z_n\|_{L^\infty(\Omega)^N} \le 1$.  \\Therefore this is sufficient to deduce the existence of $z \in L^\infty(\Omega)^N$ such that, up to subsequences, $z_n$ converges *$-$weakly in $L^\infty(\Omega)^N$ to $z$, as $n$ tends to infinity; in particular by weak lower semi-continuity, it holds
$\|z\|_{L^\infty(\Omega)^N} \le 1$.	
\medskip
\\Let us focus on \eqref{= per Gamma e z}. It follows from \cite[Proposition 4.3]{BOP} that \begin{equation}\label{= per z_n e Gamma_n}
	- \operatorname{div}\left(z_n \ensuremath{\mathrm e}^{- \Gamma_n(u_n)}\right)=h(u_n)f \ensuremath{\mathrm e}^{-\Gamma_n(u_n)} \, \text{as measures in $\Omega$,}
\end{equation}
as a consequence, \eqref{= per Gamma e z} is shown once we have taken $n$ to infinity in \eqref{= per z_n e Gamma_n}.
\\We pick $0\le\varphi \in C^1_c(\Omega)$ as a test function \black in \eqref{= per z_n e Gamma_n}, for the left-hand side we have
\begin{equation}\label{lim z_n e Gamma_n}
	\lim_{n \to \infty}\int_{\Omega} z_n \cdot \nabla \varphi \, \ensuremath{\mathrm e}^{-\Gamma_n(u_n)} = \int_{\Omega} z \cdot \nabla \varphi \,\ensuremath{\mathrm e}^{-\Gamma(u)},  
\end{equation}
which simply follows since $z_n \rightharpoonup z$ *$-$weakly in $L^\infty(\Omega)^N$ and $\ensuremath{\mathrm e}^{-\Gamma_n(u_n)} \to \ensuremath{\mathrm e}^{-\Gamma(u)}$ strongly in $L^1(\Omega)$ through the Lebesgue Theorem.
Now it remains to show that 
\begin{equation}\label{lim h(u_n)f Gamma_n}
	\lim_{n \to \infty} \int_{\Omega} h(u_n)f \ensuremath{\mathrm e}^{-\Gamma_n(u_n)} \varphi = \int_{\Omega} h(u)f \ensuremath{\mathrm e}^{-\Gamma(u)} \varphi.
\end{equation}
\medskip

We distinguish two cases: if $h$ is finite at the origin then the passage to limit is a simple consequence of the Lebesgue Theorem. Hence, without losing \black generality, we assume $h(0)=\infty$. \\We split the integral as follows \black
\begin{equation}\label{lim diviso}
	\int_{\Omega} h(u_n)f \ensuremath{\mathrm e}^{-\Gamma_n(u_n)} \varphi = \int_{\{u_n \le \delta\}}  h(u_n)f \ensuremath{\mathrm e}^{-\Gamma_n(u_n)} \varphi + \int_{\{u_n > \delta\}}  h(u_n)f \ensuremath{\mathrm e}^{-\Gamma_n(u_n)} \varphi,
\end{equation}
where $\delta \not \in \{\eta \, : \, |\{u=\eta\}|>0\}$ which is at most a countable set.
\\We highlight that $u>0$ almost everywhere in $\Omega$ since $h(0)=\infty$ and $h(u)f \in L^1_{\rm{loc}}(\Omega)$, hence $\{u=0\} \subseteq \{f=0\}$ which is Lebesgue negligible because $f>0$ a.e.
\\Applying twice the Lebesgue Theorem on the second term of right-hand side of \eqref{lim diviso},we gain
\begin{equation}\label{lim delta>0}
	\lim_{\delta \to 0} \lim_{n \to \infty} \int_{\{u_n > \delta\}} h(u_n)f \ensuremath{\mathrm e}^{-\Gamma_n(u_n)} \varphi = \lim_{\delta \to 0} \int_{\{u >\delta\}} h(u)f \ensuremath{\mathrm e}^{-\Gamma(u)} \varphi \stackrel{u>0}{=} \int_{\Omega}h(u)f \ensuremath{\mathrm e}^{-\Gamma(u)}\varphi,
\end{equation}
since $$ \chi_{\{u_n >\delta\}}h(u_n)f \ensuremath{\mathrm e}^{-\Gamma_n(u_n)} \varphi \le h_\delta(\infty)f \varphi \in L^1(\Omega),
$$
and $$ \chi_{\{u>\delta\}} h(u)f \ensuremath{\mathrm e}^{-\Gamma(u)} \varphi \le h(u)f \varphi \in L^1(\Omega).
$$
Now we prove \begin{equation}\label{lim del<0}
	\lim_{\delta \to 0} \limsup_{n \to \infty} \int_{\{u_n \le \delta\}}  h(u_n)f \ensuremath{\mathrm e}^{-\Gamma_n(u_n)} \varphi =0,
\end{equation}
which yields \black \eqref{lim h(u_n)f Gamma_n} in virtue of \eqref{lim delta>0}.
\\We fix $V_\delta(u_n) \ensuremath{\mathrm e}^{-\Gamma_n(u_n)}  \varphi\in BV(\Omega) \cap L^\infty(\Omega)$ with $0 \le \varphi \in C^1_c(\Omega)$ and $V_\delta$ is the truncation function defined in \eqref{def v d} as a test function \black in \eqref{sol approx 1}. It is an admissible choice since $\ensuremath{\mathrm e}^{-\Gamma_n(u_n)}, V_\delta(u_n)  \in BV(\Omega)\cap L^\infty(\Omega)$ and $BV(\Omega)\cap L^\infty(\Omega)$ is an algebra (see for instance \cite[Remark 3.10]{AFP}). We obtain \begin{equation}\label{tante test}
	\begin{aligned}
		&\int_\Omega (z_n, DV_\delta(u_n)) \ensuremath{\mathrm e}^{-\Gamma_n(u_n)} \varphi+ \int_\Omega \left(z_n, D\ensuremath{\mathrm e}^{-\Gamma_n(u_n)}\right) V_\delta(u_n) \varphi + \int_\Omega z_n \cdot \nabla \varphi \, \ensuremath{\mathrm e}^{-\Gamma_n(u_n)}V_\delta(u_n) \\ + &\int_\Omega g_n(u_n)|Du_n| \ensuremath{\mathrm e}^{-\Gamma_n(u_n)} V_\delta(u_n) \varphi = \int_\Omega h(u_n) f \ensuremath{\mathrm e}^{-\Gamma_n(u_n)} V_\delta(u_n) \varphi,
	\end{aligned}
\end{equation}
where we also use \eqref{GG form}, \eqref{der prod} and \eqref{sol approx 3}.
\\We underline that for a solution $u_n$ of the problem \eqref{prob approx} found in \cite[Theorem 4.4]{BOP}  by virtue of \cite[Lemma 4.7]{BOP}, it holds
\begin{equation*}
	\left(z_n, D\left(-\ensuremath{\mathrm e}^{-\Gamma_n(u_n)}\right)\right)=|D\ensuremath{\mathrm e}^{-\Gamma_n(u_n)}| \quad \text{as measures in $\Omega$}.
\end{equation*}
As direct consequence of the previous equality and the fact that $D^j u_n =0$, it yields \black 
\begin{equation}\label{catena con g_n}
	\int_\Omega \left(z_n, D\ensuremath{\mathrm e}^{-\Gamma_n(u_n)}\right) V_\delta(u_n) \varphi = - \int_{\Omega} |D\Gamma_n(u_n)| \ensuremath{\mathrm e}^{-\Gamma_n(u_n)} V_\delta(u_n) \varphi.
\end{equation}
Moreover by \eqref{sol approx 2} and $D^j u_n=0$, we get
\begin{equation}\label{V' <0}
	\begin{aligned}
		&\int_\Omega (z_n, DV_\delta(u_n)) \ensuremath{\mathrm e}^{-\Gamma_n(u_n)} \varphi = \int_{\Omega} V'_\delta(u_n) \left(z_n, DT_{2 \delta}(u_n)\right) \ensuremath{\mathrm e}^{-\Gamma_n(u_n)} \varphi = \int_{\Omega}V'_\delta(u_n) |DT_{2 \delta}(u_n)| \ensuremath{\mathrm e}^{-\Gamma_n(u_n)} \varphi \le 0,
	\end{aligned}
\end{equation}
where we also use Lemma \ref{chain rule} and \eqref{GG form}.
\\Putting \eqref{catena con g_n} and \eqref{V' <0} in \eqref{tante test}, we have
\begin{equation*}
	\int_{\{u_n \le \delta\}} h(u_n)f \ensuremath{\mathrm e}^{-\Gamma_n(u_n)} \varphi \le \int_\Omega z_n \cdot \nabla \varphi \, \ensuremath{\mathrm e}^{-\Gamma_n(u_n)} V_\delta(u_n).
\end{equation*}
Through the *$-$weak convergence of $z_n \rightharpoonup z$ in $L^\infty(\Omega)^N$ and the Lebesgue Theorem, it follows that
\begin{equation*}
	\lim_{\delta \to 0} \limsup_{n \to \infty} \int_{\{u_n \le \delta\}} h(u_n)f \ensuremath{\mathrm e}^{-\Gamma_n(u_n)} \varphi \le \lim_{\delta \to 0} \lim_{n \to \infty} \int_\Omega z_n \cdot \nabla \varphi \ensuremath{\mathrm e}^{-\Gamma_n(u_n)} V_\delta(u_n)\le \int_{\{u=0\}} z \cdot \nabla \varphi \ensuremath{\mathrm e}^{-\Gamma(u)} \stackrel{u>0}{=}0,
\end{equation*}
which demonstrates \eqref{lim del<0}.
Moreover putting together \eqref{lim z_n e Gamma_n} and \eqref{lim h(u_n)f Gamma_n}, it yields \black  \eqref{= per Gamma e z}.
\medskip
\\Now we show \eqref{dis per z}, it is enough to take limit as $n$ tends to infinity in \eqref{sol approx 1}. Firstly we analyze the left-hand side, the first integral take limit due to the *$-$weak convergence of the vector field $z_n$. For the second one, we have $$\int_{\Omega} |D\Gamma(u)| \varphi \le \liminf_{n \to \infty} \int_{\Omega} |D\Gamma_n(u_n)| \varphi, \quad \text{for all $0 \le\varphi \in C^1_c(\Omega)$,}$$ due to Lemma \ref{lemma stima} and Lemma \ref{sci BV}.
\\For the right-hand side we reason analogously as in the proof of \black \eqref{lim h(u_n)f Gamma_n}, indeed we have that $(\delta < 1)$ 
\begin{equation}\label{h(u)f che mi serve}
	\lim_{\delta \to 0} \limsup_{n \to \infty} \int_{\{u_n \le \delta\}} h(u_n)f \varphi \le \lim_{\delta \to 0} \limsup_{n \to \infty} \frac{1}{\ensuremath{\mathrm e}^{-\Gamma(1)}} \int_{\{u_n \le \delta\}} h(u_n)f \ensuremath{\mathrm e}^{-\Gamma_n(u_n)} \varphi=0,
\end{equation}
this concludes the proof of \eqref{dis per z}.
\\As consequence of \eqref{dis per z}, it gives that $z \in \mathcal{DM}^{\infty}_{\rm{loc}}(\Omega)$.
\medskip 
\\Finally, we prove $D^ju =0$.  We highlight that $J_{\Gamma(u)}=J_u$ because $\Gamma$ is an increasing function since $g$ is positive, in particular we use the pairing introduced in \eqref{mis con grat} with $\beta(s)=-\ensuremath{\mathrm e}^{-s}$ and $v=\Gamma(u)$, we obtain 
\begin{equation}\label{catena di dis}
	\begin{aligned}
		\left(\ensuremath{\mathrm e}^{-\Gamma(u)}\right)^{\#} |D \Gamma (u)| &\stackrel{\eqref{dis per z}}{\le} \ensuremath{\mathrm e}^{- \Gamma(u)} h(u)f + \left(\ensuremath{\mathrm e}^{- \Gamma(u)}\right)^{\#} \operatorname{div}z \stackrel{\eqref{= per Gamma e z}}{=}  - \operatorname{div}\left(z \ensuremath{\mathrm e}^{- \Gamma(u)}\right) + \left(\ensuremath{\mathrm e}^{- \Gamma(u)} \right)^{\#} \operatorname{div}z \\&\stackrel{\eqref{dist1}}{=}\left(z, D \left(-\ensuremath{\mathrm e}^{-\Gamma(u)}\right)^{\#}\right) \stackrel{\eqref{finitetotal1}}{\le} \left|D \left(-\ensuremath{\mathrm e}^{-\Gamma(u)}\right)\right| = \left(\ensuremath{\mathrm e}^{-\Gamma(u)}\right)^{\#}|D \Gamma(u)| \quad \text{in $\mathcal{D}'(\Omega)$,}
	\end{aligned}
\end{equation}
where the last equality is a consequence of Lemma \ref{chain rule}.
Therefore, the previous implies
$$\left(z, D \left(-\ensuremath{\mathrm e}^{-\Gamma(u)}\right)^\#\right)=\left|D\left(-\ensuremath{\mathrm e}^{-\Gamma(u)}\right)\right| \quad \text{as measures in $\Omega$}.$$
As a direct consequence of the previous equality we can deduce that $D^j \Gamma(u) =0$ through \cite[Lemma 2.3]{BOP} with $\alpha(s)=s$, $\beta(s)=- \ensuremath{\mathrm e}^{-s}$, $w=\Gamma(u)$ and \eqref{dis per z}. \\Especially, we can affirm also that $D^j u=0$, because $\Gamma$ is increasing since $g$ is positive.
\end{proof}
\subsection{Proof of the  main existence result}
Let us prove Theorem \ref{teo princ} as a consequence of the previous Lemmas.

\begin{proof} [Proof of Theorem \ref{teo princ}]
Let us consider $u_n$ a solution of the problem \eqref{prob approx} given in \cite[Theorem 4.4]{BOP}.
\\From Corollary \ref{cor es}, one has that there exists $u \in BV(\Omega)$ such that, up to subsequence, $u_n$ converges almost everywhere in $\Omega$ at $u$ as $n$ tends to infinity. Moreover in Lemma \ref{lemma dis} we found that $h(u)f \in L^1_{\rm{loc}}(\Omega)$ and $D^j u=0$. In addition, in the same lemma, we proved the existence of a vector field $z \in \mathcal{DM}^{\infty}_{\rm{loc}}(\Omega)$ such that $\|z\|_{L^\infty(\Omega)^N}\le 1$.
\medskip
\\To complete the proof, we verify that $u$ and $z$ satisfy \eqref{sol 1}, \eqref{sol 2}, \eqref{sol 3} and $g(u) \in L^1_{\rm{loc}}(\Omega, |Du|)$.
\\Let us focus on proving \eqref{sol 1}, in \eqref{catena di dis} all the inequalities become equalities, consequently we conclude that
\begin{equation*}
	\ensuremath{\mathrm e}^{-\Gamma(u)} \left(- \operatorname{div}z+|D\Gamma(u)|\right)=\ensuremath{\mathrm e}^{-\Gamma(u)}h(u)f \quad \text{in $\mathcal{D}'(\Omega)$,}
\end{equation*}
recalling that $\ensuremath{\mathrm e}^{-\Gamma(u)}>0$ $\mathcal{H}^{N-1}$-a.e. $x \in \Omega$ thanks to \eqref{BV finita} because $\Gamma(u) \in BV(\Omega)$, as proven in Corollary \ref{cor es},  we can deduce 
\begin{equation}\label{quasi sol 1}
	- \operatorname{div}z+|D\Gamma(u)|=h(u)f \quad \text{in $\mathcal{D}'(\Omega)$.}
\end{equation}
\black
Now using \cite[Lemma 2.4]{GOP}, we can affirm that
$$\chi_{\{u>0\}}|D\Gamma(u)|=g(u)\chi_{\{u>0\}}|Du| \quad \text{as measures in $\Omega$,}$$
by Lemma \ref{lemma dis} and \cite[Remark 2.5]{GOP}, we deduce
\begin{equation}\label{spac Gamma}
	|D\Gamma(u)|=g(u)|Du| \quad \text{as measures in $\Omega$,}
\end{equation}
which implies that $g(u) \in L^1(\Omega, |Du|)$.
\\Putting together \eqref{quasi sol 1} and \eqref{spac Gamma}, we can conclude that holds \eqref{sol 1}. We highlight that $z \in \DM(\Omega)$ since it holds Lemma \ref{lemma estensione}.
\medskip
\\Our next goal is proving \eqref{sol 2}.
We take $T_k(u_n) \varphi \in BV(\Omega) \cap L^\infty(\Omega)$ as a test function in \eqref{sol approx 1} with $0\le \varphi \in C^1_c(\Omega)$, which is an admissible choice as stated in Remark \ref{oss estensione}, we have
\begin{equation}\label{verso sol 2}
	\int_\Omega |DT_k(u_n)| \varphi + \int_\Omega z_n \cdot \nabla \varphi T_k(u_n) + \int_\Omega g_n(u_n)|Du_n|T_k(u_n)\varphi= \int_\Omega h(u_n)f T_k(u_n) \varphi,
\end{equation}
where we use \eqref{Leibniz} and \eqref{sol approx 2}.
\\To reach our objective, we take limit as $n$ tends to infinity in \eqref{verso sol 2}. We name $$\tilde{\Gamma}_n(s):=\int_0^s T_k(\sigma) g_n(\sigma) \, \ensuremath{\mathrm d}\sigma,$$ hence the first and the third integral on left hand side of \eqref{verso sol 2}, pass to limit as $n \to \infty$ by Lemma \ref{sci BV} (we underline that $T_k(u_n) \to T_k(u)$ and $\tilde{\Gamma}_n(u_n) \to \tilde{\Gamma}(u)$ strongly in $L^1(\Omega)$ as a consequence of Corollary \ref{cor es}), instead the second integral take limit as $n \to \infty$ because $z_n \rightharpoonup z$ $*$-weakly in $L^\infty(\Omega)^N$ as stated in Lemma \ref{lemma dis} and $T_k(u_n) \to T_k(u)$ strongly in $L^1(\Omega)$ as proven in Corollary \ref{cor es}.
\\Finally, for the right hand side we reason analogously as in the proof of \eqref{lim h(u_n)f Gamma_n}, indeed we have that $(\delta<1)$
\begin{equation*}
	\lim_{\delta \to 0} \limsup_{n \to \infty} \int_{\{u_n\le \delta\}} h(u_n)f T_k(u_n) \varphi \le k \lim_{\delta \to 0} \limsup_{n \to \infty} \int_{\{u_n\le \delta\}} h(u_n)f  \varphi \stackrel{\eqref{h(u)f che mi serve}}{=}0.
\end{equation*}
Hence, \eqref{verso sol 2} becomes
\begin{equation}\label{quasi sol 222}
	\int_\Omega |DT_k(u)| \varphi + \int_\Omega z \cdot \nabla \varphi T_k(u) + \int_\Omega g(u)|Du|T_k(u)\varphi \le \int_\Omega h(u)f T_k(u)\varphi,
\end{equation}
where we use $$|D\tilde{\Gamma}(u)|=g(u)|Du|T_k(u) \quad \text{as measures in $\Omega$,}$$ as a consequence of \cite[Remark 2.5]{GOP}, $|\{u=0\}|=0$ and $D^ju=0$.
\\Recall Remark \ref{oss estensione}, we can expand the space of test function of \eqref{sol 1} to $BV(\Omega)\cap L^\infty(\Omega)$, hence from \eqref{quasi sol 222}, we can deduce that
\begin{equation*}
	\int_\Omega |DT_k(u)|\varphi \le -\int_\Omega z \cdot \nabla \varphi T_k(u)  - \int_\Omega T_k(u) \varphi \operatorname{div}z \stackrel{\eqref{dist1}}{=} \int_\Omega \left(z, DT_k(u)\right) \varphi,
\end{equation*} 
which means \begin{equation*}
	|DT_k(u)| \le (z,DT_k(u)) \quad \text{as measures in $\Omega$, for any $k>0$,}
\end{equation*}
and, being the reverse inequality trivial, this proves \eqref{sol 2}. \black

	\medskip
	Now it remains to show \eqref{sol 3}. In \cite[Lemma 4.9]{BOP} it is proven, for every $k>0$, that
	\begin{equation}\label{= al bordo}
		\int_{\partial \Omega} \left([T_k(u_n) z_n, \nu]+T_k(u_n)\right) \, \ensuremath{\mathrm d}\mathcal{H}^{N-1} + \int_{\partial \Omega} \tilde{\Gamma}_n(u_n) \, \ensuremath{\mathrm d}\mathcal{H}^{N-1}=0.
	\end{equation}
	
	In particular, it follows from \eqref{des3} and $\|z_n\|_{L^\infty(\Omega)^N}\le 1$ that
	\begin{equation}\label{dis al bordo}
		|[T_k(u_n) z_n, \nu]| \le T_k(u_n) \, \text{on $\partial \Omega$.}
	\end{equation}
	Therefore gathering \eqref{dis al bordo} into \eqref{= al bordo}, one yields \black \begin{equation*}
		\int_{\partial \Omega} \tilde{\Gamma}_n(u_n) \, \ensuremath{\mathrm d} \mathcal{H}^{N-1} = 0,
	\end{equation*}
	using the Fatou Lemma we gain 
	\begin{equation*}
		\int_{\partial \Omega} \tilde{\Gamma}(u) \, \ensuremath{\mathrm d} \mathcal{H}^{N-1} \le 0,
	\end{equation*} 
	so $\tilde{\Gamma}(u) = 0$ $\mathcal{H}^{N-1}-$a.e. on $\partial \Omega$, as consequence it gives \eqref{sol 3}. 
	\medskip
	\\Finally if $\|f\|_{L^N(\Omega)} < (\mathcal{S}_1 h(\infty))^{-1}$, from Corollary \ref{cor es}, we can deduce that $u \in L^\infty(\Omega)$. This concludes the proof.
\end{proof}
\section{The case of a nonnegative $f \in L^N(\Omega)$}
In this section we extend to the case of a nonnegative datum $f \in L^N(\Omega)$. Let consider
\begin{equation}\label{f nonne}
	\begin{cases}
		-\Delta_1 u +g(u)|Du|=h(u)f & \text{in $\Omega$,} \\
		u=0 & \text{on $\partial \Omega$,}
	\end{cases}
\end{equation}
with $g,h$ satisfying \eqref{g in zero} - \eqref{h g infinito}; in particular we assume that $h(0)=\infty$.
\medskip
\\In this case, as the solution is not expected to be positive, the notion of solution needs to be suitably adapted.
\begin{defin}\label{sol nonne}
	Let $0 \le f \in L^N(\Omega)$. A nonnegative $u \in BV(\Omega)$ is a solution to the problem \eqref{f nonne} if $\chi_{\{u>0\}} \in BV_{\rm{loc}}(\Omega)$, $D^ju=0$, $\chi_{\{u>0\}}g(u) \in L^1_{\rm{loc}}(\Omega, |Du|)$, $h(u)f \in L^1_{\rm{loc}}(\Omega)$, and there exists $z \in \DM_{\rm{loc}}(\Omega)$ with $\|z\|_{L^\infty(\Omega)^N}\le 1$ such that
	\begin{equation}\label{+ sol 1}
		-\chi_{\{u>0\}}^* \operatorname{div}z + \chi_{\{u>0\}}g(u)|Du|=h(u)f \quad \text{as measures in $\Omega$,}
	\end{equation}
	\begin{equation}\label{+ sol 2}
		(z, DT_k(u))=|DT_k(u)| \quad \text{in $\mathcal{D}'(\Omega)$ for any $k>0$,}
	\end{equation}
	and
	\begin{equation}\label{+ sol 3}
		u(x)=0 \quad \text{for $\mathcal{H}^{N-1}-$a.e. in $\partial \Omega$.}
	\end{equation}
\end{defin}
\begin{remark}
	Let us highlight the main difference with respect to Definition \ref{sol}, i.e. the presence of the characteristic function $\chi_{\{u>0\}}$ in \eqref{+ sol 1}, because we cannot deduce that $u>0$ a.e. in $\Omega$ from $h(u)f \in L^1_{\rm{loc}}(\Omega)$ because the set $\{f =0\}$ is not Lebesgue negligible.				
\end{remark}
\begin{theorem}
	Let $g$ be positive and satisfying \eqref{g in zero}, \eqref{g non nulla} and \eqref{h g infinito}. Let $h$ satisfy \eqref{h in zero} and \eqref{h g infinito} such that $0 \le \theta < 1$, $0 \le \gamma \le 1$ and $\theta + \gamma \le 1$ and let $0 \le f \in L^N(\Omega)$. Then there exists a solution to the problem \eqref{f nonne} in the sense of Definition \ref{sol nonne}.
	\\Moreover if $\|f\|_{L^N(\Omega)} < (\mathcal{S}_1 h(\infty))^{-1}$, then $u \in L^\infty(\Omega)$.
\end{theorem}
\begin{proof}
	We only emphasize the few differences \black with the proof of Theorem \ref{teo princ}.
	\medskip
	\\Let us introduce the following approximation scheme
	\begin{equation}\label{nuovo prob tronc}
		\begin{cases}
			-\Delta_1 u_n  + g(u_n)|Du_n|=h(u_n)f_n & \text{in $\Omega$,}\\
			u_n=0 &\text{on $\partial\Omega$,}
		\end{cases}
	\end{equation}
	where $f_n(x):=\max\{\frac{1}{n}, f(x)\}$ for all $x \in \Omega$.
	Thanks to Theorem \ref{teo princ} there exists $z_n \in \DM_{\rm{loc}}(\Omega)$ and a positive $u_n \in BV(\Omega)$ such that $D^ju_n=0$, $g(u_n) \in L^1(\Omega, |Du_n|)$, $h(u_n)f_n \in L^1_{\rm{loc}}(\Omega)$ such that
	$$-\operatorname{div}z_n+g(u_n)|Du_n|=h(u_n)f_n \quad \text{as measures in $\Omega$,}$$
	$$(z_n,DT_k(u_n))=|DT_k(u_n)| \quad \text{in $\mathcal{D}'(\Omega)$ and for any $k>0$,}$$
	and
	$$u_n=0 \quad \text{$\mathcal{H}^{N-1}-$a.e. on $\partial \Omega$.}$$ 
	Initially, we want to prove \eqref{+ sol 1}, we repeat the same arguments of Lemma \ref{lemma stima} and Corollary \ref{cor es}, it is possible because we can extend the space of test functions to $BV(\Omega) \cap L^\infty(\Omega)$ thanks to Remark \ref{oss estensione}.
	\medskip
	\\After, we take limit as $n$ tends to infinity in the equation \eqref{= per Gamma e z}, which is
	$$-\operatorname{div}\left(z_n \ensuremath{e}^{-\Gamma(u_n)}\right)=h(u_n)f_n \ensuremath{e}^{-\Gamma(u_n)} \quad \text{in $\mathcal{D}'(\Omega)$.}$$
	The first term passes to the limit by the Lebesgue Theorem. For the second term we split the integral into two parts
	$$\int_{\Omega}h(u_n)f_n\ensuremath{\mathrm e}^{-\Gamma(u_n)} \varphi=\int_{\{u_n > \delta\}} h(u_n)f_n\ensuremath{\mathrm e}^{-\Gamma(u_n)} \varphi+ \int_{\{u_n \le \delta\}}h(u_n)f_n\ensuremath{\mathrm e}^{-\Gamma(u_n)}\varphi,$$
	with $\delta \not \in \{\eta : |\{u=\eta\}|>0\}$ which is at most a countable set and $0 \le \varphi \in C^1_c(\Omega)$.
	\\The first quantity passes to the limit as $n$ tends to infinity and after as $\delta$ tends to $0$ applying twice the Lebesgue Theorem, thus we gain
	\begin{equation}\label{lim con f nonne}
		\lim_{\delta \to 0} \lim_{n \to \infty} \int_{\{u_n > \delta\}} h(u_n)f_n\ensuremath{\mathrm e}^{-\Gamma(u_n)} \varphi = \int_{\{u>0\}} h(u)f \ensuremath{\mathrm e}^{-\Gamma(u)} \varphi=\int_\Omega h(u)f \ensuremath{\mathrm e}^{-\Gamma(u)} \varphi,
	\end{equation}
	where the last equality is guaranteed by the fact that $h(u)f \in L^1_{\rm{loc}}(\Omega)$, hence $\{u=0\} \subseteq \{f=0\}$. 
	Instead for the second integral, it holds
	\begin{equation}\label{dis con f nonne}
		0\le \int_{\{u_n \le \delta\}}h(u_n)f_n\ensuremath{\mathrm e}^{-\Gamma(u)}\varphi \quad \text{for every $n \in \mathbb{N}, \delta>0$.}
	\end{equation}
	Hence putting together \eqref{lim con f nonne} and \eqref{dis con f nonne}, we obtain
	\begin{equation}\label{dis per z Gamma f nonne}
		-\operatorname{div}\left(z \ensuremath{\mathrm e}^{-\Gamma(u)}\right) \ge h(u)f \ensuremath{\mathrm e}^{-\Gamma(u)} \quad \text{as measures in $\Omega$.}
	\end{equation}
	
	\medskip
	Now we take $\left(1-V_\delta(u_n)\right)\varphi \in BV(\Omega)\cap L^\infty(\Omega)$ with $0\le \varphi \in C^1_c(\Omega)$ as a test function \black in \eqref{nuovo prob tronc}, it yields \black
	\begin{equation*}
		\int_{\Omega}z_n \cdot \nabla \varphi \left(1-V_\delta(u_n)\right)+\int_\Omega \left(z_n, D(1-V_\delta (u_n))\right)\varphi + \int_\Omega |D\Gamma(u_n)|\left(1-V_\delta(u_n)\right)\varphi=\int_\Omega h(u_n)f_n \left(1-V_\delta(u_n)\right)\varphi,
	\end{equation*} 
	where we want to pass to limit as $n$ tends to infinity and $\delta$ tends to $0$.
	\\The first integral on the left-hand side passes to limit because $z_n \rightharpoonup z$ *$-$weakly in $L^\infty(\Omega)^N$ and $1-V_\delta(u_n) \to \chi_{\{u>0\}}$ strongly in $L^1(\Omega)$ by the Lebesgue Theorem, hence we gain
	\begin{equation}\label{primo}
		\lim_{\delta \to 0}\lim_{n \to \infty} \int_{\Omega}z_n \cdot \nabla \varphi \left(1-V_\delta(u_n)\right) = \int_\Omega z \cdot \nabla \varphi \chi_{\{u>0\}}.
	\end{equation}
	For the second integral on the left-hand side, we recall that
	$$\left(z_n, D(1-V_\delta(u_n))\right)=\left(z_n, D(1-V_\delta(T_{2\delta}(u_n)))\right) \quad \text{as measures in $\Omega$},$$
	using \cite[Lemma 2.3]{GMP} we get
	\begin{equation*}
		\left(z_n, D(1-V_\delta(T_{2\delta}(u_n)))\right)=\left|D\left(1-V_\delta(T_{2\delta}(u_n))\right)\right|=|D(1-V_\delta(u_n))|,
	\end{equation*}
	therefore applying twice Lemma \ref{sci BV}, we can affirm
	\begin{equation}\label{secondo}
		\int_\Omega |D\chi_{\{u>0\}}|\varphi \le \liminf_{\delta \to 0} \liminf_{n \to \infty} \int_\Omega |D(1-V_\delta(u_n))|\varphi,
	\end{equation}
	
	which implies that $\chi_{\{u>0\}} \in BV_{\rm{loc}}(\Omega)$ because $u_n$ are uniformly bounded in $BV(\Omega)$ with respect to $n \in \mathbb{N}$. \black
	Now we analyze the third integral on the left-hand side. Firstly we introduce the function $\hat{\Gamma}_{\delta}:[0, \infty) \to [0, \infty)$ such that
	$$\hat{\Gamma}_{\delta}(s):=\int_0^s g(t)(1-V_\delta(t)) \, \ensuremath{\mathrm d}t,$$
	hence using twice Lemma \ref{sci BV}, it yields \black 
	\begin{equation}\label{terzo}
		\int_\Omega |D\Gamma(u)|\varphi \le \liminf_{\delta \to 0} \liminf_{n \to \infty} \int_\Omega |D\hat{\Gamma}_{\delta}(u_n)| \varphi,
	\end{equation}
	because we highlight that
	$$\int_0^s g(t)\chi_{(0,\infty)}(t) \, \ensuremath{\mathrm d}t=\int_0^s g(t) \, \ensuremath{\mathrm d}t.$$
	Finally, the integral on the right-hand side passes to the limit through the Lebesgue Theorem, it holds
	\begin{equation}\label{destra}
		\lim_{\delta \to 0}\lim_{n \to \infty} \int_\Omega h(u_n)f_n \left(1-V_\delta(u_n)\right)\varphi =\int_\Omega h(u)f\varphi,
	\end{equation}
	where we remember that $h(u)f \in L^1_{\rm{loc}}(\Omega)$ and $\{u=0\} \subseteq \{f=0\}$.
	\\Putting together \eqref{primo}, \eqref{secondo}, \eqref{terzo} and \eqref{destra}, we can conclude that
	\begin{equation}\label{dis preli f non}
		-\operatorname{div}\left(\chi_{\{u>0\}}z\right)+\left|D\chi_{\{u>0\}}\right| +|D\Gamma(u)| \le h(u)f \quad \text{as measures in $\Omega$.}
	\end{equation}
	Since $\|z\|_{L^\infty(\Omega)^N} \le 1$, we observe that
	\begin{equation}\label{z con chi}
		-\chi_{\{u>0\}}^*\operatorname{div}z\stackrel{\eqref{Leibniz}}{=}-\operatorname{div}\left(\chi_{\{u>0\}}z\right)+\left(z, D\chi_{\{u>0\}}\right)\le-\operatorname{div}\left(\chi_{\{u>0\}}z\right)+|D\chi_{\{u>0\}}| \quad \text{as measures in $\Omega$,}
	\end{equation}
	as a consequence putting \eqref{z con chi} into \eqref{dis preli f non}, we can affirm that
	\begin{equation}\label{dis z f nonne quasi sol}
		-\chi_{\{u>0\}}^*\operatorname{div}z+|D\Gamma(u)|\le h(u)f \quad \text{as measures in $\Omega$.}
	\end{equation}
	\medskip
	Now we show that $D^j u=0$. We observe that
	\begin{equation*}
		\begin{aligned}
			(\ensuremath{\mathrm e}^{-\Gamma(u)})^\#\chi_{\{u>0\}}^*|D\Gamma(u)| &\stackrel{\eqref{dis z f nonne quasi sol}}{\le} \ensuremath{\mathrm e}^{-\Gamma(u)} h(u)f \chi_{\{u>0\}}+ (\ensuremath{\mathrm e}^{-\Gamma(u)})^\#  \chi_{\{u>0\}}^* \operatorname{div}z \\ &\stackrel{\eqref{dis per z Gamma f nonne}}{\le} -\operatorname{div}\left(z \ensuremath{\mathrm e}^{-\Gamma(u)} \right) \chi_{\{u>0\}}^*+(\ensuremath{\mathrm e}^{-\Gamma(u)})^\# \chi_{\{u>0\}}^*\operatorname{div}z \\ & \stackrel{\eqref{mis con grat}}{=} \chi_{\{u>0\}}^* \left(z, D\left(-\ensuremath{\mathrm e}^{-\Gamma(u)}\right)^\#\right) \\ & \stackrel{\eqref{dis mis con grat}}{\le} \chi_{\{u>0\}}^*|D\ensuremath{\mathrm e}^{-\Gamma(u)}|\stackrel{\eqref{eq chain}}{=}(\ensuremath{\mathrm e}^{-\Gamma(u)})^\# \chi_{\{u>0\}}^*|D\Gamma(u)|,
		\end{aligned}
	\end{equation*}
	which means that
	\begin{equation}\label{= per no jump}
		\chi_{\{u>0\}}^*\left(z, D(-\ensuremath{\mathrm e}^{-\Gamma(u)})^\#\right)=\chi_{\{u>0\}}^*|D\ensuremath{\mathrm e}^{-\Gamma(u)}| \quad \text{as measures in $\Omega$.}
	\end{equation}
	Now we use a variation of \cite[Lemma 2.3]{BOP}: we observe that $\chi_{\{u>0\}}^*>0$ $\mathcal{H}^{N-1}-$a.e. on $J_{\Gamma(u)}$, hence from \eqref{= per no jump} and \eqref{dis z f nonne quasi sol} we can deduce that $D^j u=0$,
	where we highlight that $\{u=0\}\subseteq \{f=0\}$ because $h(u)f \in L^1_{\rm{loc}}(\Omega)$ and $\Gamma(u)=0$ where $u=0$.
	From now on, we repeat step-by-step the same arguments of Theorem \ref{teo princ}, formulas \eqref{423} involving $\chi_{\{u>0\}}^*$, hence it holds $$-\chi_{\{u>0\}}^*\operatorname{div}z+|D\Gamma(u)|=h(u)f \quad \text{as measures in $\Omega$,}$$
	thus applying \cite[Lemma 2.4]{GOP} we deduce \eqref{+ sol 1}.
	\medskip
	\\Finally, we observe that \eqref{+ sol 2} and \eqref{+ sol 3} follow as in Theorem \ref{teo princ} and this concludes the proof. 
\end{proof}
\section*{Data availability}
No data was used for the research described in the article.

\section*{Acknowledgment}
F. Balducci is partially supported by the Gruppo Nazionale per l’Analisi Matematica, la Probabilità e le loro Applicazioni (GNAMPA) of the Istituto Nazionale di Alta Matematica (INdAM).
\\Finally, we wish to thank the anonymous referee for carefully reading this article and providing some useful comments.

\end{document}